\documentclass{amsart}
\usepackage{verbatim}
\usepackage[textsize=scriptsize]{todonotes}
\usepackage{longtable}

\usepackage{etoolbox}
\newtoggle{draft}
\newtoggle{final}

\usepackage{tikz}
\usetikzlibrary{matrix,arrows}
\usepackage{tikz-cd}
\usepackage{stmaryrd}
\usepackage{listings}

\toggletrue{final}

\usepackage{etoolbox}
\iftoggle{draft} {
\usepackage[margin=1.5in]{geometry}
}{ 
\usepackage[margin=1in]{geometry}
\setlength{\marginparwidth}{0.75in}
}
\usepackage{amsmath}
\usepackage{amssymb}
\usepackage{amsthm}
\usepackage{amscd}
\usepackage{enumerate}
\usepackage[pdfusetitle,unicode,hidelinks]{hyperref}
\usepackage{bbm}
\usepackage{etoolbox}

\usepackage[utf8]{inputenc}
\usepackage[T1]{fontenc}

\newcommand{\tomemail}{\href{mailto:tom.bachmann@zoho.com}{tom.bachmann@zoho.com}}

\newtheorem{proposition}{Proposition}
\newtheorem{corollary}[proposition]{Corollary}
\newtheorem{lemma}[proposition]{Lemma}
\newtheorem{theorem}[proposition]{Theorem}

\newtheorem*{conjecture*}{Conjecture}
\newtheorem*{theorem*}{Theorem}
\newtheorem*{corollary*}{Corollary}
\newtheorem*{proposition*}{Proposition}
\newtheorem*{lemma*}{Lemma}
\theoremstyle{definition}
\newtheorem{definition}[proposition]{Definition}

\newtheorem*{definition*}{Definition}
\newtheorem*{construction*}{Construction}
\theoremstyle{remark}
\newtheorem{remark}[proposition]{Remark}
\newtheorem*{remark*}{Remark}

\newtheorem{example}[proposition]{Example}
\newtheorem*{example*}{Example}

\newcommand{\id}{\operatorname{id}}
\newcommand{\Z}{\mathbb{Z}}
\def\C{\mathbb C}
\newcommand{\N}{\mathbb{N}}
\newcommand{\Q}{\mathbb{Q}}
\newcommand{\F}{\mathbb{F}}

\let\scr=\mathcal
\let\bb=\mathbb
\newcommand{\Gm}{{\mathbb{G}_m}}
\newcommand{\Gmp}[1]{{\mathbb{G}_m^{\wedge #1}}}
\def\A{\bb A}
\def\P{\bb P}
\def\R{\bb R}
\newcommand{\1}{\mathbbm{1}}

\newcommand{\eff}{{\text{eff}}}
\newcommand{\veff}{{\text{veff}}}

\newcommand{\SH}{\mathcal{SH}}

\DeclareMathOperator*{\colim}{colim}

\let\lim=\relax
\DeclareMathOperator*{\lim}{lim}
\def\Map{\mathrm{Map}}

\def\PSh{\mathcal{P}}

\def\Spc{\mathcal{S}\mathrm{pc}{}}
\def\Fin{\cat F\mathrm{in}}
\def\Fun{\mathrm{Fun}}
\def\red{\mathrm{red}}
\newcommand{\Spec}{\mathrm{Spec}}
\newcommand{\gp}{\mathrm{gp}}

\newcommand{\wequi}{\simeq}

\def\adj{\leftrightarrows}

\DeclareRobustCommand{\ul}{\underline}
\newcommand{\heart}{\heartsuit}

\def\op{\mathrm{op}}

\let\cat=\mathrm
\def\Sm{{\cat{S}\mathrm{m}}}

\def\Nis{\mathrm{Nis}}

\def\mot{\mathrm{mot}}
\newcommand{\et}{{\acute{e}t}}
\newcommand{\ret}{{r\acute{e}t}}
\newcommand{\fr}{\mathrm{fr}}

\newcommand{\lra}[1]{\langle #1 \rangle}
\def\ph{\mathord-}

\numberwithin{proposition}{section}
\numberwithin{equation}{section}

\iftoggle{final} {
\renewcommand{\todo}[1]{}
\newcommand{\NB}[1]{}
}{ 

\newcommand{\NB}[1]{\todo[color=gray!40]{#1}}
}

\newcommand{\fpsr}[1]{\llbracket #1 \rrbracket}

\newcommand{\Sper}{\mathrm{Sper}}

\newcommand{\KO}{\mathrm{KO}}

\newcommand{\ko}{\mathrm{ko}}

\newcommand{\kw}{\mathrm{kw}}
\newcommand{\KW}{\mathrm{KW}}
\newcommand{\HW}{\mathrm{HW}}
\newcommand{\HZ}{\mathrm{H}\Z}
\newcommand{\MSL}{\mathrm{MSL}}

\newcommand{\MSp}{\mathrm{MSp}}

\newcommand{\SL}{\mathrm{SL}}

\newcommand{\W}{\mathrm{W}}
\newcommand{\I}{\mathrm{I}}

\newcommand{\Cor}{\mathrm{Cor}}

\newcommand{\SHSfr}{\SH^{S^1\fr}}
\newcommand{\Shv}{\mathcal{S}\mathrm{hv}}
\newcommand{\HZW}{\mathrm{H}_W\Z}
\newcommand{\HtildeZ}{\mathrm{H}\tilde\Z}

\newcommand{\comp}{{{\kern -.5pt}\wedge}}
\newcommand{\vcd}{\mathrm{vcd}}
\newcommand{\cd}{\mathrm{cd}}

\newcommand{\fib}{\mathrm{fib}}

\newcommand{\adamsphi}{\varphi}
\newcommand{\adamspsi}{\psi}

\title{\texorpdfstring{$\eta$}{eta}-periodic motivic stable homotopy theory over Dedekind domains}
\date{\today}

\author{Tom Bachmann}
\address{Mathematisches Institut, LMU Munich, Munich, Germany}
\email{\tomemail}

\begin{document}

\maketitle

\begin{abstract}
We construct well-behaved extensions of the motivic spectra representing generalized motivic cohomology and connective Balmer--Witt $K$-theory (among others) to mixed characteristic Dedekind schemes on which $2$ is invertible.
As a consequence we lift the fundamental fiber sequence of $\eta$-periodic motivic stable homotopy theory established in \cite{bachmann-eta} from fields to arbitrary base schemes, and use this to determine (among other things) the $\eta$-periodized algebraic symplectic and $\SL$-cobordism groups of mixed characteristic Dedekind schemes containing $1/2$.
\end{abstract}

\setcounter{tocdepth}{1}
\tableofcontents

\todo{image of j spectrum?}
\section{Introduction} \label{sec:intro}
Let $k$ be a field.
We have spectra \[ \KO, \KW, \kw, \HW, \HZW, \HtildeZ \in \SH(k) \] representing interesting cohomology theories for smooth $k$-varieties: $\KO$ represents hermitian $K$-theory, $\KW$ represents Balmer--Witt $K$-theory, $\HW$ represents cohomology with coefficients in the sheaf of Witt groups, $\HtildeZ$ represents higher Chow--Witt groups; $\kw, \HZW$ are more technical but have featured prominently in e.g. \cite{bachmann-eta,bachmann-very-effective}.
The first aim of this article is to define extensions of these spectra to other bases.
The utility of such extensions is manifold; e.g. they can be used in integrality arguments \cite{bachmann-euler}.

Thus let $D$ be the prime spectrum of a Dedekind domain, perhaps of mixed characteristic, or a field.
Consider the motivic spectrum $\KO_D \in \SH(D)$ (see e.g. \cite[\S\S2.2, 4.1]{bachmann-norms} for a definition of the motivic stable category $\SH(D)$), representing Hermitian $K$-theory \cite{hornbostel2005a1}.
From this we can build the following related spectra
\begin{align*}
\KW_D &= \KO_D[\eta^{-1}] \\
\kw_D &= \tau_{\ge 0} \KW_D \\
\HW_D &= \tau_{\le 0} \kw_D \\
\HZW_D &= f_0(\HW_D) \\
\ul{K}^W_D &= \tau_{\le 0} \HZW_D \\
\HtildeZ_D &= \tau_{\le 0}^\eff(\1_D).
\end{align*}
Here by $\tau_{\le 0}, \tau_{\ge 0}$ we mean the truncation in the homotopy $t$-structure on $\SH(D)$, and by $\tau_{\le 0}^\eff$ we mean the truncation in $\SH(D)^\eff$ (see e.g. \cite[\S B]{bachmann-norms}).
If $f: D' \to D$ is any morphism there are natural induced \emph{base change maps} $f^*(\kw_D) \to \kw_{D'}$, and so on.\NB{$f^*(\kw) \in \SH(D')_{\ge 0}$ so $f^*(\kw) \to f^*(\KW) \to \KW$ factors through $\kw_{D'} \to \KW_{D'}$. This corresponds to $\kw \to f_*(\kw)$ yielding $\kw \to f_*(\HW)$. Since $f_*(\HW) \in \SH(D)_{\le 0}$ this factors through $\HW \to f_*(\HW)$, corresponding to $f^*(\HW) \to \HW$. Etc.}
It thus makes sense to ask if the spectra above are stable under base change, i.e., if the base change maps are equivalences.
This is true for $\KO$ (and $\KW$), since this spectrum can be built out of (orthogonal or symplectic) Grassmannians \cite{schlichting2015geometric,panin2010motivic}, which are stable under base change.
Our main result is the following.

\begin{theorem}[see Theorem \ref{thm:main-1}] \label{thm:intro-main}
All of the above spectra are stable under base change among Dedekind domains or fields, provided that they contain $1/2$.
\end{theorem}
Over fields, the above definitions of spectra coincide with other definitions that can be found in the literature (see \cite{bachmann-very-effective,Deglise17,ananyevskiy2017very,A1-alg-top}; this is proved in Lemma \ref{lemm:defn-compat}).
In other words, we construct well-behaved extensions to motivic stable homotopy theory over Dedekind domains of certain motivic spectra which so far have mainly been useful over fields.
In fact, we show that all of the above spectra (which we have built above out of $\KO$ by certain universal properties) admit more explicit (and so calculationally useful) descriptions.
For example we show that \[ \ul{\pi}_0(\ul{K}^W)_* = \ul{I}^*, \quad \ul{\pi}_i(\ul{K}^W)_* = 0 \text{ for } i \ne 0; \] here $\ul{I}$ is the Nisnevich sheaf associated with the presheaf of fundamental ideals in the Witt rings.
\begin{remark}
The above description of $\ul\pi_* \ul{K}^W_*$ asserts in particular that the sheaf $\ul{W}$ is strictly $\A^1$-invariant.
In fact variants of this property form a starting point of our proofs, and are the reason for assuming that $S$ is Dedekind.
Unwinding the arguments, one finds that we ultimately rely on the Gersten conjecture (for étale cohomology of essentially smooth schemes over discrete valuation rings) via \cite{geisser2004motivic}.
\end{remark}
\begin{remark}
Using recent results on Gersten resolutions \cite{deshmukh2020nisnevich}, our results may be extended to regular J2 schemes instead of just Dedekind schemes.
Alternatively, using the cdh topology instead of the Nisnevich one, they may be extended to all schemes.
These facts will be recorded elsewhere.
\end{remark}

Our motivating application of these results is as follows.
Using Theorem \ref{thm:intro-main}, together with the fact that equivalences (and connectivity) of motivic spectra over $D$ can be checked after pullback to the residue fields of $D$ \cite[Proposition B.3]{bachmann-norms}, one obtains essentially for free the following extension of \cite{bachmann-eta}.
\begin{corollary}[see Theorem \ref{thm:main2}]
For $D$ as above, there is a fiber sequence \[ \1[\eta^{-1}]_{(2)} \to \kw_{(2)} \to \Sigma^4 \kw_{(2)} \in \SH(D). \]
\end{corollary}
Using this, we also extend many of the other results of \cite{bachmann-eta} to Dedekind domains.

\subsection*{Overview}
The main observation allowing us to prove the above results is the following.
Recall that there is an equivalence $\SH(D) \wequi \SH^\fr(D)$, where the right hand side means the category of motivic spectra with framed transfers \cite{EHKSY,hoyois2018localization}.
This supplies us with an auxiliary functor $\sigma^\infty_\fr: \SHSfr(D) \to \SH^\fr(D)$.
The Hopf map $\eta: \Gm \to \1$ already exists in $\SHSfr(D)$ (see \S\ref{subsec:eta}).
This readily implies that we can make sense of the category $\SHSfr(D)[\eta^{-1}]$ of $\eta$-periodic $S^1$-spectra with framed transfers, and that there is an equivalence \[ \SHSfr(D)[\eta^{-1}] \wequi \SH^\fr(D)[\eta^{-1}]. \]
The significance of this is that the left hand side no longer involves $\P^1$-stabilization, and hence is much easier to control.
In the end this allows us to relate all our spectra in the list above to a spectrum $\ko^\fr$ which is known to be stable under base change.
To do so we employ (1) work of Jeremy Jacobson \cite{jacobson2018cohomological} on the Gersten conjecture for Witt rings in mixed characteristic, and (2) work of Markus Spitzweck \cite{spitzweck2012commutative} on stability under base change of $\HZ$.

\subsection*{Organization}
In \S\ref{sec:In} we construct by hand a motivic spectrum $\ul{K}^W$ with the expected homotopy sheaves.
In \S\ref{sec:eta-periodic-framed} we study some truncations in $\SHSfr(D)[\eta^{-1}]$, allowing us among other things to construct a spectrum $\kw$ with the expected homotopy sheaves.
We prove our main theorems in \S\ref{sec:main}.
We first give alternative, more explicit definitions of the spectra in our list and deduce stability under base change.
Then we show that the spectra we constructed satisfy the expected universal properties.
We establish the fundamental fiber sequence of $\eta$-periodic motivic stable homotopy theory as an easy corollary.
Finally in \S\ref{sec:app} we deduce some applications, mostly in parallel with \cite[\S8]{bachmann-eta}.

\subsection*{Notation and terminology}
By a Dedekind scheme we mean a finite disjoint union of spectra of Dedekind domains or fields, that is, a regular noetherian scheme of Krull dimension $\le 1$\NB{e.g. wikipedia, Dedekind domain, (DD4)}.
Given a non-vanishing integer $n$ and a scheme $X$, we write $1/n \in X$ to mean that $n \in \scr O_X(X)^\times$.

We denote by $\Spc(S) \subset \PSh(\Sm_S)$ the $\infty$-category of motivic spaces, that is, the subcategory of motivically local (i.e. $\A^1$-invariant and Nisnevich local) presheaves.
We write $L_\mot$ for the left adjoint of the inclusion, i.e., the motivic localization functor.
For a motivic spectrum $E \in \SH(S)$ we denote by $\ul{\pi}_i(E)_j$ the homotopy sheaves (see e.g. \cite[\S2.4.2]{bachmann-eta}).
Beware that unless the base is a field, these objects are only loosely related to the homotopy $t$-structure.

We denote by $a_\Nis, a_\et,$ and $a_\ret$ respectively the associated sheaves of sets in the Nisnevich, étale and real étale topologies.
We write $L_\Nis$ for the Nisnevich localization of presheaves of spaces or spectra.
Unless specified otherwise, all cohomology is with respect to the Nisnevich topology.

All schemes are assumed quasi-compact and quasi-separated.

We denote by $\Spc$ the $\infty$-category of spaces, and by $\SH$ the $\infty$-category of spectra.

\subsection*{Acknowledgements}
I would like to thank Shane Kelly for help with Lemma \ref{lemm:dvr-pres}.
To the best of my knowledge, the first person suggesting to study $\ko^\fr$ was Marc Hoyois.

\section{The sheaves $\ul{I}^n$} \label{sec:In}
\subsection{}
For a scheme $X$ (with $1/2 \in X$), denote by $\ul{W}$ the Nisnevich sheaf of commutative discrete rings obtained by sheafification from the presheaf of Witt rings \cite[\S I.5]{knebusch-bilinear}.
The canonical map \[ \ul{W} \to a_\et \ul{W} \wequi \ul{\Z/2} \] is the \emph{rank map}, and its kernel is the ideal sheaf $\ul{I} \subset \ul{W}$.
We write $\ul{I}^*$ for the sheaf of commutative graded rings given by the powers of $\ul{I}$.
Somewhat anachronistically we put \[ \ul{k}_n^M = a_\Nis H^n_\et(\ph, \Z/2); \] this is also a sheaf of commutative graded rings.
Note that since $1/2 \in X$ we have an exact sequence of étale sheaves $0 \to \Z/2 \to \scr O^\times \to \scr O^\times \to 1$, yielding a boundary map $\scr O^\times(X) \to H^1_\et(X, \Z/2)$ which we denote by $a \mapsto (a)$.
The following results justifies our notation $\ul{k}_*^M$ to an extent.

\begin{theorem}[Jacobson \cite{jacobson2018cohomological}] \label{thm:jac1}
Assume that $1/2 \in X$.
Then there is a unique map of sheaves (of rings) $\ul{I}^* \to \ul{k}_*^M$ given in degree zero by the rank and in degree one locally by $(\lra{a} - 1) \mapsto (a)$.
This map annihilates $\ul{I}^{*+1} \subset \ul{I}^*$ and induces an isomorphism of sheaves \[ \ul{I}^*/\ul{I}^{*+1} \wequi \ul{k}_*^M. \]
\end{theorem}
\begin{proof}
For existence, see \cite[Remark 4.5]{jacobson2018cohomological}.
The rest is \cite[Theorem 4.4]{jacobson2018cohomological}.\NB{actually reduce via henselization to fields; see \cite[Lemma 4.1]{jacobson2018cohomological}}
\end{proof}

Still assuming that $1/2 \in X$, the canonical map \[ \sigma: \ul{W} \to a_\ret \ul{W} \wequi a_\ret \Z \] is the \emph{global signature}.
One may show that $\sigma(\ul{I}) \subset 2a_\ret \Z$ (indeed locally $\I(\ph)$ consists of diagonal forms of even rank \cite[Corollary I.3.4]{milnor1973symmetric}, and the signature is thus a sum of an even number of terms $\pm 1$) and hence $\sigma(\ul{I}^n) \subset 2^n a_\ret \Z$.
Since $a_\ret \Z$ is torsion-free, there are thus induced maps \[ \sigma/2^n: \ul{I}^n \to a_\ret \Z. \]

For a scheme $X$, denote by $K(X)$ the product of the residue fields of its minimal points.
Recall that for a field $k$, $\vcd_p(k)$ denotes the minimum of $\cd_p(l)$ for $l/k$ a finite extension, and $\cd_p(k)$ denotes the $p$-étale cohomological dimension (see e.g. \cite[\S I.3.1]{serre2013galois}).
\begin{lemma} \label{lemm:vcd-bound} \NB{this seems a little bit fishy, but I checked it several times}
Let $X$ be a noetherian scheme with $1/p \in X$.
Then $\vcd_p(X) \le \dim{X} + \vcd_p(K(X))$.
\end{lemma}
\begin{proof}
If $p=2$, replace $X$ by $X[\sqrt{-1}]$.
We may thus assume that $\vcd_p = \cd_p$, and if $p=2$ that all residue fields of $X$ are unorderable.
By \cite[Lemma XVIII-A.2.2]{illusie2014travaux} we have $\cd_p(X) \le \dim{X} + \sup_{x \in X}(\dim \scr O_{X,x} + \cd_p(k(x)))$.
It hence suffices to show that $\cd_p(k(x)) + \dim \scr O_{X,x} \le \cd_p(K(X))$.
Since $\cd_p(K(X)) \ge \cd_p(K(\scr O_{X,x}))$, this follows from \cite[Corollary X.2.4]{sga4}.
\end{proof}

\begin{proposition}[Jacobson] \label{prop:jac2}
Assume that $X$ is noetherian and $1/2 \in X$.
Then for $n > \vcd_2(K(X))+\dim{X}$\NB{$n > \vcd_2(K(X))$ should be enough} the divided signature \[ \sigma/2^n: \ul{I}^n \to a_\ret \Z \] is an isomorphism of sheaves on $X_\Nis$.
\end{proposition}
\begin{proof}
Let $N = \vcd_2(K(X)) + \dim{X}$.
Note that for any Hensel local ring $A$ of $X$ we have $\vcd_2(A) \le N$, by Lemma \ref{lemm:vcd-bound} and \cite[Theorem X.2.1]{sga4}.
Since $A$ is henselian and noetherian, by \cite[Lemma 6.2(III)]{jacobson2018cohomological} we have $\cap_n I^n(A) = 0$.
Hence by \cite[Corollary 4.8]{jacobson2018cohomological} for $n>N$ the map $I^n(A) \xrightarrow{2} I^{n+1}(A)$ is an isomorphism.
By \cite[Proposition 7.1]{jacobson-fundamental-ideal}, the divided signatures induce an isomorphism $\colim_n I^n(A) \wequi (a_\ret \Z)(A)$.
These two results imply that $\sigma/2^n(A): I^n(A) \to (a_\ret \Z)(A)$ is an isomorphism, for any $n>N$.
Since $A$ was arbitrary, the map $\sigma/2^n: \ul{I}^n \to a_\ret \Z$ induces an isomorphism on stalks, and hence is an isomorphism.
\end{proof}

Combining Theorem \ref{thm:jac1} and Proposition \ref{prop:jac2}, we will be able to control the sheaves $\ul I^n$ by controlling $a_\ret \Z$ and $\ul k_*^M$.

\subsection{}
Fix a Dedekind scheme $D$.
A $\Gm$-prespectrum $E$ over $D$ means a sequence of objects $(E_1, E_2, \dots)$ with $E_i \in \Fun(\Sm_D^\op, \SH)$, together with maps $E_i \to \Omega_\Gm E_{i+1}$.
Such a prespectrum can in particular be exhibited by defining $E_i$ as a presheaf of abelian groups.
See e.g. \cite[\S6]{cisinski2009local} for details as well as symmetric (monoidal) variants.
A $\Gm$-prespectrum $E$ is called a \emph{motivic spectrum} if each $E_i$ is motivically local, and the structure maps $E_i \to \Omega_\Gm E_{i+1}$ are equivalences.

\begin{example}[Spitzweck \cite{spitzweck2012commutative}]
There is a $\Gm$-prespectrum $\HZ/2$ with \[ (\HZ/2)_i = \Sigma^i \tau_{\ge -i}^\Nis L_\et \Z/2. \]
In particular \[ \ul{\pi}_*(\HZ/2)_* \wequi \ul{k}_*^M[\tau], \] where $\tau \in \ul\pi_1(\HZ/2)_1(D)$ is the unique non-vanishing element.
The prespectrum $\HZ/2$ is in fact a motivic spectrum.
\end{example}
\begin{example} \label{ex:build-spectrum-from-ring}
Let $R_*$ be a Nisnevich sheaf of commutative graded (discrete) rings, and $t \in R_1(\A^1 \setminus 0)$.
Then $R_*$ defines a commutative monoid $\tilde R_*$ in symmetric sequences (of Nisnevich sheaves) with trivial symmetric group actions\NB{$\Fin \to \N$ is symmetric monoidal, restrict to cores and form lax symmetric monoidal right adjoint on presheaf categories}, and $t$ defines a class $[t]$ in the summand $R_1(\Gm)$, making $\tilde R_*$ into a commutative monoid under the free commutative monoid on $\Gm$.
In other words, $\tilde R_*$ is a commutative monoid in symmetric $\Gm$-prespectra \cite[second half of \S6.6]{cisinski2009local}.
This construction is functorial in $R_*$.
\end{example}

\begin{definition}
Applying Example \ref{ex:build-spectrum-from-ring} to the sheaf of graded rings $\ul{I}^*$ and the class $\lra{t}-1 \in I(D \times \Gm)$, we obtain a $\Gm$-prespectrum $\ul{K}^W$ over $D$ with $\ul{K}^W_n = \ul{I}^n$.
Similarly we obtain a $\Gm$-prespectrum $\ul{k}^M$, and in fact a morphism of commutative monoids in symmetric $\Gm$-prespectra $\ul{K}^W \to \ul{k}^M$ (coming from the ring map $\ul{I}^* \to \ul{k}_*^M$ of Theorem \ref{thm:jac1}).
\end{definition}
From now on we view the category of Nisnevich sheaves of abelian groups as embedded into Nisnevich sheaves of spectra, and view all sheaves of abelian groups as sheaves of spectra, so that for $X \in \Sm_D$ we have \[ (\ul{K}^W)_i(X) = L_\Nis \ul{I}^i, \] and similarly for $\ul{k}^M$.

\begin{lemma} \label{lemm:kM-mot-local}
There is a commutative ring map $\HZ/2 \to \ul{k}^M$ inducing an equivalence of $\Gm$-prespectra $\ul{k}^M \wequi (\HZ/2)/\tau$.
In particular $\ul{k}^M$ is a motivic spectrum.
\end{lemma}
\begin{proof}
Let $E = (E_1, E_2, \dots)$ be a $\Gm$-prespectrum in Nisnevich sheaves of spectra.
If each $E_i$ is connective, we can form a prespectrum $\tau_{\le 0}^\Nis(E)$ with $\tau_{\le 0}^\Nis(E)_i \wequi \tau_{\le 0}^\Nis(E_i)$ the truncation in the usual $t$-structure, and bonding maps\NB{Can also do it when $E_i$ is not connective, since $\Omega_{\Gm}$ is exact.} \[ \Gm \wedge \tau_{\le 0}^\Nis(E_i) \to \tau_{\le 0}^\Nis(\Gm \wedge \tau_{\le 0}^\Nis(E_i)) \wequi \tau_{\le 0}^\Nis(\Gm \wedge E_i) \to \tau_{\le 0}^\Nis(E_{i+1}). \]
Even if $E$ is a motivic spectrum $\tau_{\le 0}^\Nis(E)$ need not be; however if it is then it represents the truncation $\tau_{\le 0}(E) \in \SH(D)$ in the homotopy $t$-structure.

In \cite[\S4.1.1]{spitzweck2012commutative} there is a construction of a specific $\Gm$-prespectrum $H$ such that (1) $H$ is a motivic spectrum representing $\HZ/2$ and (2) $\tau_{\le 0}^\Nis(H) \wequi \ul{k}^M$, the equivalence being as $\Gm$-prespectra.
The map $\HZ/2 \to (\HZ/2)/\tau \in \SH(D)$ corresponds to a map $H \to H'$ of $\Gm$-spectra which is immediately seen to be a levelwise zero-truncation.
It follows that $H' \wequi \tau_{\le 0}^\Nis(H) \wequi \ul{k}^M$ as $\Gm$-prespectra.
In particular $\tau_{\le 0}^\Nis(H) \wequi \ul{k}^M$ are motivic spectra, and in fact $\ul{k}^M \wequi \tau_{\le 0}(\HZ/2) \in \SH(D)$.
Since $\SH(D)_{\ge 0}$ is closed under smash products, truncation in the homotopy $t$-structure is lax symmetric monoidal on $\SH(D)_{\ge 0}$ and so $\tau_{\le 0}(\HZ/2)$ admits a canonical ring structure making $\HZ/2 \to (\HZ/2)_{\le 0}$ into a commutative ring map.
It remains to show that $\tau_{\le 0}(\HZ/2) \wequi \ul{k}^M$ is an equivalence of ring spectra.
Both of them can be modeled by $\scr E_\infty$-monoids in the ordinary $1$-category of symmetric $\Gm$-prespectra of sheaves of abelian groups on $\Sm_D$; i.e. just commutative monoids in the usual sense.
The isomorphism between them preserves the product structure by inspection.
\end{proof}

The following is the main result of this section.
\begin{corollary} \label{cor:cons-K-W}
Let $D$ be a Dedekind scheme with $1/2 \in D$.
The $\Gm$-prespectrum $\ul{K}^W$ is a motivic spectrum over $D$.
\end{corollary}
\begin{proof}
Since $(\ul{K}^W)_n = L_\Nis \ul{I}^n$ is Nisnevich local by construction, to prove it is motivically local we need to establish $\A^1$-homotopy invariance, i.e. that $\Omega_{\A^1_+} \ul{K}^W_n \wequi \ul{K}^W_n$.
Similarly to prove that we have a spectrum we need to show that $\Omega_{\Gm}\ul{K}^W_{n+1} \wequi \ul{K}^W_n$.
Here we are working in the category $\Shv^\Nis_\SH(\Sm_D)$ of Nisnevich sheaves of spectra on $\Sm_D$.
For $x \in D$, denote by $p_x: D_x \to D$ the inclusion of the local scheme.
By \cite[Lemmas A.3 and A.4]{hoyois-algebraic-cobordism}, the functor $p_x^*$ commutes with $\Omega_{\Gm}$ and $\Omega_{\A^1_+}$, and by \cite[Proposition A.3(1,3)]{bachmann-norms} the family of functors $p_x^*: \Shv^\Nis_\SH(\Sm_D) \to \Shv^\Nis_\SH(\Sm_{D_x})$ is conservative.
Finally by \cite[Corollary 51]{bachmann-gwtimes} we have $p_x^* \ul{I}^n \wequi \ul{I}^n$.
It follows that we may assume (replacing $D$ by $D_x$) that $D$ is the spectrum of a discrete valuation ring or field.

By Lemma \ref{lemm:dvr-pres} below, we have $D \wequi \lim_\alpha D_\alpha$, where each $D_\alpha$ is the spectrum of a discrete valuation ring or field and $\vcd_2(K(D_\alpha)) < \infty$.
By \cite[Theorem 8.8.2(ii), Proposition 17.7.8(ii)]{EGAIV} for $X \in \Sm_D$ there exists (possibly after shrinking the indexing system) a presentation $X \wequi \lim_\alpha X_\alpha$, with $X_\alpha \in \Sm_{D_\alpha}$ and the transition maps being affine.
\todo{more explicit references for the following?}
We have a fibered topos \cite[\S Vbis.7]{sga4} $X_{\bullet\Nis}$ with $\lim X_{\bullet\Nis} \wequi X_\Nis$.
The sheaves $(\ul{I}^n|_{X_{\alpha\Nis}})_\alpha$ define a section of the fibered topos $X_{\bullet\Nis}$.
It follows from \cite[Lemma 49]{bachmann-gwtimes} and \cite[Proposition Vbis.8.5.2]{sga4} that (in the notation of the latter reference) $Q^*(\ul{I}^n|_{X_{\bullet\Nis}}) \wequi \ul{I}^n|_{X_\Nis}$.
Hence by \cite[Theorem Vbis.8.7.3]{sga4} we get \[ H^i(X, \ul{I}^n) \wequi \colim_\alpha H^i(X_\alpha, \ul{I}^n). \]
The same holds for cohomology on $X \times \A^1$ and $X_+ \wedge \Gm$.
We may thus assume (replacing $D$ by $D_\alpha$) that $\vcd_2(K(D)) < \infty$.

Let $X \in \Sm_D$.
We need to prove that $(*)$ \[ H^*(X, \ul{I}^n) \wequi H^*(X \times \A^1, \ul{I}^n) \wequi H^*(X_+ \wedge \Gm, \ul{I}^n). \]
Since $\ul{k}^M_n$ satisfies the analog of $(*)$ by Lemma \ref{lemm:kM-mot-local}, the exact sequence $\ul{I}^{n+1} \to \ul{I}^n \to \ul{k}_n^M$ from Theorem \ref{thm:jac1} shows that $(*)$ holds for $\ul{I}^n$ if and only if it holds for $\ul{I}^{n+1}$.
By Proposition \ref{prop:jac2} (and \cite[Theorem X.2.1]{sga4}), for $n$ sufficiently large we get $\ul{I}^n|_{X_\Nis} \wequi a_\ret \Z|_{X_\Nis}$.
It thus suffices to show that $L_\Nis a_\ret \Z$ satisfies the analog of $(*)$.
This follows from the fact that there exists a motivic spectrum $E = H_{\A^1} \Z[\rho^{-1}]$ with $E_i = L_\Nis a_\ret \Z$ for all $i$ \cite[Proposition 41]{bachmann-real-etale}.
\end{proof}
\begin{remark} \label{rmk:W-contraction}
Theorem \ref{thm:jac1} shows that $\Omega_\Gm L_\Nis \ul{I} \xrightarrow{\wequi} \Omega_\Gm L_\Nis \ul{W}$; indeed the cofiber of this map is given by $\Omega_\Gm \ul{\Z/2} \wequi 0$ (the vanishing holds e.g. since motivic cohomology vanishes in negative weights).
It follows that $\ul{\pi}_0(\ul{K}^W)_* \wequi \ul{I}^*$ and $\ul{\pi}_i = 0$ for $i \ne 0$; here $\ul{I}^n = \ul{W}$ for $n < 0$.
It also follows that $\Omega_\Gm L_\Nis \ul{W} \wequi L_\Nis \ul{W}$, and that this spectral sheaf is homotopy invariant.
\end{remark}

Recall that a noetherian valuation ring is a ring which is either a discrete valuation ring or a field \cite[Tag 00II]{stacks-project}.\NB{just a convenient abbreviation...}
\begin{lemma} \label{lemm:dvr-pres}
Let $R$ be a noetherian valuation ring.
Then there is a filtered system $R_\alpha$ of noetherian valuation rings with $\vcd(K(R_\alpha)) < \infty$ (i.e. there exists $N$ with $\vcd_p < N$ for all $p$) and $\colim_\alpha R_\alpha \wequi R$.
\end{lemma}
\begin{proof}
Let $K = K(R)$.
Then $K = \colim_\alpha K_\alpha$, where the colimit is over finitely generated subfields $K_\alpha \subset K$; this colimit is filtered.
Let $R_\alpha = R \cap K_\alpha$.
We shall show that $R_\alpha$ is a noetherian valuation ring, and $K(R_\alpha) = K_\alpha$.
This will imply the result since $\vcd_p(K_\alpha) < \infty$ uniformly in $p$ by \cite[Theorem X.2.1, Proposition X.6.1, Theorem X.5.1]{sga4}.
It is clear that $R_\alpha$ is a valuation ring and $K(R_\alpha) = K_\alpha$: if $x \in K_\alpha^\times$ then one of $x, x^{-1} \in R$ \cite[Tag 00IB]{stacks-project}, and hence one of $x, x^{-1} \in R_\alpha$; thus we conclude by \cite[Tag 052K]{stacks-project}.
To show that $R_\alpha$ is noetherian we must show that $K_\alpha^\times/R_\alpha^\times \wequi \Z$ or $\wequi 0$ \cite[Tags 00IE and 00II]{stacks-project}.
This is clear since there is an injection $K_\alpha^\times/R_\alpha^\times \hookrightarrow K^\times/R^\times$ and the latter group is $\wequi \Z$ or $\wequi 0$.
\end{proof}

\section{\texorpdfstring{$\eta$}{eta}-periodic framed spectra} \label{sec:eta-periodic-framed}
\subsection{}
Recall from \cite[\S3.2.2]{EHKSY} the symmetric monoidal $\infty$-category $\Cor^\fr(S)$ under $\Sm_S$.
We denote by $\Spc^\fr(S)$ and $\SHSfr(S)$ the motivic localizations of respectively $\Fun(\Cor^\fr(S)^\op, \Spc)$ and $\Fun(\Cor^\fr(S)^\op, \SH)$, and we put $\SH^\fr(S) = \SHSfr(S)[\Gmp{-1}] \wequi \Spc^\fr(S)[(\P^1)^{-1}]$.
The free-forgetful adjunction \[ \SH(S) \adj \SH^\fr(S) \] is an adjoint equivalence \cite[Theorem 18]{hoyois2018localization}.
We obtain the diagram
\begin{equation*}
\begin{tikzcd}
\Spc^\fr(S) \ar[r, "\Sigma^\infty_{S^1}", bend left=10] \ar[rr, "\Sigma^\infty_\fr" swap, bend left=50] & \SHSfr(S) \ar[r, "\sigma^\infty_\fr", bend left=10] \ar[l, "\Omega^\infty_{S^1}", bend left=10] & \SH(S). \ar[l, "\omega^\infty_\fr", bend left=10] \ar[ll, "\Omega^\infty_\fr" swap, bend left=50]
\end{tikzcd}
\end{equation*}

\subsection{}
Put \[ \Shv^\Nis_\SH(\Sm_S) = L_\Nis \Fun(\Sm_S^\op, \SH) \quad\text{and}\quad \Shv^\Nis_\SH(\Cor^\fr(S)) = L_\Nis \Fun(\Cor^\fr(S)^\op, \SH). \]
On either category we consider the $t$-structure with non-negative part generated \cite[Proposition 1.4.4.11]{lurie-ha} by the smooth schemes.

\begin{lemma} \label{lemm:SHSfr-basics}
\begin{enumerate}
\item $E \in \Fun(\Cor^\fr(S)^\op, \SH)$ is Nisnevich local (or homotopy invariant, or motivically local) if and only if the underlying spectral presheaf $UE \in \Fun(\Sm_S^\op, \SH)$ is.
\item The forgetful functor $U: \Shv^\Nis_\SH(\Cor^\fr(S)) \to \Shv^\Nis_\SH(\Sm_S)$ is $t$-exact.
\end{enumerate}
\end{lemma}
\begin{proof}
(1) holds by definition.
(2) The functor $U_\Sigma: \PSh_\Sigma(\Cor^\fr(S)) \to \PSh_\Sigma(\Sm_S)_*$ preserves filtered (in fact sifted) colimits and commutes with $L_\Nis$ \cite[Proposition 3.2.14]{EHKSY}.
Consequently $U_\Nis: \Shv^\Nis(\Cor^\fr(S)) \to \Shv_\Nis(\Sm_S)_*$ also preserves filtered colimits.
Being a right adjoint it also preserves limits, and hence commutes with spectrification.
Consequently it suffices to prove the following: given $F \in \PSh_\Sigma(\Cor^\fr(S))$ and $n \ge 0$, we have $U_\Sigma(\Sigma^n F) \in \PSh_\Sigma(\Sm_S)_{\ge n}$; indeed then \[ U(\Sigma^\infty F) \wequi \colim_n \Sigma^{\infty-n} U_\Sigma(\Sigma^n F) \in \Shv^\Nis_\SH(\Sm_S)_{\ge 0} \] by what we have already said.
Writing $\Sigma^n F$ as an iterated sifted colimit, using semi-additivity of $\PSh_\Sigma(\Cor^\fr(S))$ \cite[sentence after Lemma 3.2.5]{EHKSY} and the fact that $U_\Sigma$ commutes with sifted colimits, we find that $U_\Sigma(\Sigma^n F)$ is given by $B^nU_\Sigma(F)$, i.e. the iterated bar construction applied sectionwise.
The required connectivity is well-known; see e.g. \cite[Proposition 1.5]{segal1974categories}.
\end{proof}

We denote by $\tau_{\ge i}^\Nis$, $\tau_{\le i}^\Nis$ and $\tau_{=i}^\Nis$ the truncation functors corresponding to the above $t$-structures.
Note that the $t$-structure we have constructed on $\Shv^\Nis_\SH(\Sm_S)$ coincides with the usual one \cite[Definition 1.3.2.5]{lurie-sag}, and in particular $\Shv^\Nis_\SH(\Sm_S)^\heart$ is just the category of Nisnevich sheaves of abelian groups on $\Sm_S$ \cite[Proposition 1.3.2.7(4)]{lurie-sag}.

\begin{remark} \label{rmk:Sigma-infty-S1}
The proof of Lemma \ref{lemm:SHSfr-basics} also shows the following: if $F \in \Shv^\Nis(\Cor^\fr(S))$ and $X \in \Sm_S$, then $F(X) \in \Spc$ is a commutative monoid ($\Cor^\fr(S)$ being semiadditive) and $\Sigma^\infty F \in \Shv^\Nis_\SH(\Cor^\fr(S))$ has underlying sheaf of spectra corresponding to the group completion of $F$.
\end{remark}

\begin{remark} \label{rmk:htpy-sheaves-conservative}
If $S$ has finite Krull dimension, then the above $t$-structure is non-degenerate \cite[Proposition A.3]{bachmann-norms}.
\end{remark}

\subsection{} \label{subsec:eta}
The unit $t \in \scr O(\A^1 \setminus 0)^\times$ defines a framing of the identity on $\A^1 \setminus 0$ and hence a framed correspondence $\A^1\setminus 0 \rightsquigarrow *$.
We denote by \[\eta: \Gm \to \1 \in \SHSfr(S) \] the corresponding map (obtain by precomposition with $\Gm \hookrightarrow \Gm \vee \1 \wequi \A^1 \setminus 0$).

The following is a key result.
It shows that the Hopf map $\eta$ is already accessible in framed $S^1$-spectra, which enables all our further results.
\begin{lemma} \label{lemm:eta-framed}
There is a homotopy $\sigma_\fr^\infty(\eta) \wequi \eta$, where on the right hand side we mean the usual motivic stable Hopf map.
\end{lemma}
\begin{proof}
Write $\lra{t}: (\A^1 \setminus 0)_+ \to (\A^1 \setminus 0)_+ \in \SH(S)$ for the map induced by the framing $t$ of the identity, and $\tilde\eta$ for the induced map $\Sigma^\infty_+ \Gm \to \1$.
By \cite[Example 3.1.6]{EHKSY2}, the map $\lra{t}$ is given by $\Sigma^{\infty-2,1}$ of the map of pointed motivic spaces \[ (\A^1 \setminus 0)_+ \wedge \A^1/(\A^1 \setminus 0) \to (\A^1 \setminus 0)_+ \wedge \A^1/(\A^1 \setminus 0), \quad (t, x) \mapsto (t, tx), \] and hence $\tilde\eta$ is given by $\Sigma^{\infty-2,1}$ of the map \[ \eta': (\A^1 \setminus 0)_+ \wedge \A^1/(\A^1 \setminus 0) \to \A^1/\A^1 \setminus 0, \quad (t, x) \mapsto tx. \]
Consider the commutative diagram of pointed motivic spaces (with $\A^1$ pointed at $1$)
\begin{equation*}
\begin{CD}
\Gm \times \Gm @>>> \Gm \times \A^1 @>>> (\Gm \times \A^1)/(\Gm \times \Gm) \wequi (\A^1 \setminus 0)_+ \wedge \A^1/(\A^1 \setminus 0) \\
@VVV                       @VVV                @V{\eta'}VV \\
\Gm @>>>                \A^1 @>>>            \A^1/\Gm.
\end{CD}
\end{equation*}
All the vertical maps are induced by $(t, x) \mapsto tx$ and the horizontal maps are induced by the canonical inclusions and projections.
Stably this splits as \cite[Lemma 6.1.1]{morel2004motivic-pi0}
\begin{equation*}
\begin{CD}
\Gm \vee \Gm \vee \Gmp{2} @>{(\id,0,0)}>> \Gm @>>> S^{2,1} \vee \Sigma^{2,1}\Gm \\
@V{(\id,\id,\Sigma^{1,1}\eta)}VV          @VVV      @V{\Sigma^{2,1}\tilde\eta}VV \\
\Gm                       @>>>             0  @>>> S^{2,1}.
\end{CD}
\end{equation*}
Since the rows are cofibration sequences, it follows\todo{really?} that $\tilde \eta \wequi (\id, \eta)$, which implies the desired result.
\end{proof}

We call $E \in \SHSfr(S)$ \emph{$\eta$-periodic} if the canonical map $\eta^*: E \to \Omega_\Gm E$ is an equivalence.
Write \[ \SHSfr(S)[\eta^{-1}] \subset \SHSfr(S) \] for the full subcategory on $\eta$-periodic spectra.
These are the local objects of a symmetric monoidal localization of $\SHSfr(S)$.
\begin{lemma}
There is a canonical equivalence $\SHSfr(S)[\eta^{-1}] \wequi \SH(S)[\eta^{-1}]$.
\end{lemma}
\begin{proof}
In light of Lemma \ref{lemm:eta-framed}, this is a special case of \cite[Proposition 3.2]{hoyois2016equivariant}.
\end{proof}

\subsection{}
From now on we fix a Dedekind scheme $D$ with $1/2 \in D$.
Consider $\omega^\infty_\fr(\KW) \in \SHSfr(D) \subset \Shv^\Nis_\SH(\Cor^\fr(D))$.
This object is $\eta$-periodic by construction.

\begin{definition} \label{def:HW}
We put \[ \HW = \tau_{=0}^\Nis \omega^\infty_\fr(\KW) \in \Shv^\Nis_\SH(\Cor^\fr(D)). \]
\end{definition}
\begin{remark} \label{rmk:KW-htpy}
Multiplication by $\beta^i$ induces $\tau_{=4i}^\Nis \omega^\infty_\fr(\KW) \wequi \HW$, and the other homotopy sheaves vanish \cite[Proposition 6.3]{schlichting2016hermitian} \cite[Theorem 1.5.22]{balmer2005witt}.
\end{remark}

\begin{lemma} \label{HW-mot-loc}
$\HW$ is motivically local and $\eta$-periodic.
\end{lemma}
\begin{proof}
The framed construction of $\KO$ as in \cite[\S A]{bachmann-euler} together with Lemma \ref{lemm:eta-framed} shows that the map $\eta^*: \HW \to \Omega_\Gm \HW$ is on the level of spectral sheaves (i.e. sheaves of abelian groups) induced by multiplication by $\lra{t}-1$.\NB{Can arbitrarily add multiples of the identity, since they go away in the projection to $\Gm$.}
The result now follows from Remark \ref{rmk:W-contraction}.
\end{proof}

Thus $\HW$ defines an object of $\SH(D)[\eta^{-1}]$.

\subsection{} \label{subsec:KW-HW}
By construction, the forgetful functor $\SHSfr(D)[\eta^{-1}] \to \SH(D)[\eta^{-1}] \subset \SH(D)$ sends $\HW$ to the spectrum $(\ul{W}, \ul{W}, \dots)$, with the canonical structure maps.
It follows that there is a canonical morphism (of ring spectra) $\ul{K}^W \to \HW$, given in degree $n$ by the inclusion $\ul{I}^n \hookrightarrow \ul{W}$.
The main point of the following result is to determine the action of $\eta$ on $\ul{K}^W$.
\begin{lemma} \label{lemm:identify-HW}
The canonical maps $\ul{K}^W \to \HW$ and $\ul{K}^W \to \ul{k}^M$ induce equivalences $\ul{K}^W[\eta^{-1}] \wequi \HW$ and $\ul{K}^W/\eta \wequi \ul{k}^M$.
\end{lemma}
\begin{proof}
The equivalence $\SH^\fr(D) \wequi \SH(D)$ restricts to the subcategories of those objects such that $\ul{\pi}_i(\ph)_j = 0$ for $i \ne 0$.
These are exactly the objects representable by prespectra valued in sheaves of abelian groups that are motivic spectra when viewed as valued in spectral sheaves.
The construction of $\HW$ supplies the sheaf $\ul{W}$ with a structure of framed transfers.
Then the subsheaf $\ul{I}^n$ admits at most one compatible structure of framed transfers; the existence of the map $\ul{K}^W \to \HW$ together with the above discussion shows that this structure exists.
This supplies a description of the (unique) lift of $\ul{K}^W$ to $\SH^\fr(D)$.
From this and Lemma \ref{lemm:eta-framed} it follows that the action by $\eta$ on $\ul{K}^W$ induces the inclusion $\ul{I}^{*+1} \to \ul{I}^*$ on homotopy sheaves.
This implies immediately that $\ul{K}^W[\eta^{-1}] \to \HW$ induces an isomorphism on homotopy sheaves and hence is an equivalence (see Remark \ref{rmk:htpy-sheaves-conservative}).
For $\ul{K}^W/\eta$ the same argument works using Theorem \ref{thm:jac1}.
\end{proof}

\subsection{}
\begin{definition} \label{def:kw}
We put $\kw = \tau_{\ge 0}^\Nis \omega_\fr^\infty \KW \in \Shv^\Nis_\SH(\Cor^\fr(D))$
\end{definition}
The canonical map $\KO \to \KW \in \SH(D)$ induces \[ \tau_{\ge 0}^\Nis \omega_\fr^\infty \KO \to \kw. \]

The following is one of the main results of this section.
\begin{lemma} \label{lemm:identify-kw}
\begin{enumerate}
\item The objects $\tau_{\ge 0}^\Nis \omega^\infty_\fr \KO, \kw \in \Shv^\Nis_\SH(\Cor^\fr(D))$ are motivically local.
\item $\kw$ is $\eta$-periodic.
\item The canonical map $\Sigma^\infty_{S^1} \Omega^\infty_\fr \KO \to \tau_{\ge 0}^\Nis \omega^\infty_\fr \KO$ is an equivalence.
\item The canonical map $\Sigma^\infty_{S^1} \Omega^\infty_\fr \KO \to \kw$ is an $\eta$-periodization.
\end{enumerate}
\end{lemma}
\begin{proof}
(1, 2) Since the negative homotopy sheaves (in weight $0$) of $\KO$ and $\KW$ coincide \cite[Proposition 6.3]{schlichting2016hermitian}, we have \[ \tau_{< 0}^\Nis \omega^\infty_\fr \KO \wequi \tau_{< 0}^\Nis \omega^\infty_\fr \KW =: E. \]
Since $\omega^\infty \KO$ and $\omega^\infty_\fr \KW$ are motivically local, and $\omega^\infty_\fr \KW$ is $\eta$-periodic, it thus suffices to show that $E$ is motivically local and $\eta$-periodic.
Since motivically local, $\eta$-periodic spectral presheaves are closed under limits and colimits, by Remark \ref{rmk:KW-htpy} it thus suffices to show that $\tau^\Nis_{=0} \omega^\infty_\fr \KW \wequi \HW$ is motivically local and $\eta$-periodic.
This is Lemma \ref{HW-mot-loc}.

(3) It follows from Remark \ref{rmk:Sigma-infty-S1} that the functor $\Sigma^\infty_{S^1} \Omega^\infty_{S^1}$ sends the spectral sheaf $E$ on $\Cor^\fr(D)$ to $L_\mot E'$, where $E'(X) \wequi E(X)_{\ge 0}$.
We thus obtain $\Sigma^\infty_{S^1} \Omega^\infty_{S^1} \wequi L_\mot \tau_{\ge 0}^\Nis$.
Since $\Omega^\infty_\fr \wequi \Omega^\infty_{S^1} \circ \omega^\infty_\fr$, the result follows.

(4) $\KO \to \KW \in \SH(S)$ is an $\eta$-periodization.
It follows from \cite[Lemma 3.3]{hoyois2016equivariant} and \cite[Lemma 12.1]{bachmann-norms}\NB{details?} that $\omega^\infty_\fr$ preserves $\eta$-periodizations.
Hence $\omega^\infty_\fr \KO \to \omega^\infty_\fr \KW$ is an $\eta$-periodization.
Thus it suffices to show that \[ \tau^\Nis_{< 0} \omega^\infty_\fr \KO \to \tau^\Nis_{< 0} \omega^\infty_\fr \KW \] is an $\eta$-periodization.
Since this is an equivalence of $\eta$-periodic objects (see the proof of (1,2)), this is clear.
\end{proof}

Thus in particular $\kw$ defines an object of $\SH(D)[\eta^{-1}]$.

\section{Main results} \label{sec:main}
\subsection{}
Fix a Dedekind scheme $D$ with $1/2 \in D$.
In the previous two sections we have defined ring spectra $\ul{K}^W, \ul{k}^M$ and $\HW$ (and $\kw$) in $\SH(D)$.
\begin{definition}
Define commutative ring spectra $\HZW$ and $\HtildeZ$ in $\SH(D)$ as pullbacks in the following diagram with cartesian squares
\begin{equation*}
\begin{CD}
\HtildeZ @>>> \HZW @>>> \ul{K}^W \\
@VVV         @VVV       @VVV \\
\HZ @>>> \HZ/2 @>>> \ul{k}^M.
\end{CD}
\end{equation*}
\end{definition}

\begin{definition}
Put $\ko^\fr := \Sigma^\infty_\fr \Omega^\infty_\fr \KO$.
\end{definition}

\begin{lemma} \label{lemm:defn-compat}
If $D$ is the spectrum of a field, then $\ul{K}^W, \ul{k}^M, \HW, \HZW, \HtildeZ$ and $\kw$ coincide with their usual definitions, and $\ko^\fr \wequi \ko$.
\end{lemma}
Over a field, the spectra $\ul{K}^W, \ul{k}^M$ are defined in \cite[Example 3.33]{A1-alg-top}, the spectra $\HZW, \HtildeZ$ are defined in \cite[Notation p. 12]{bachmann-very-effective} (another definition of $\HtildeZ$ was given in \cite{Deglise17}, the two definitions are shown to coincide in \cite{bachmann-criterion}), the spectrum $\ko$ is defined in \cite{ananyevskiy2017very}, and the spectra $\HW, \kw$ are defined in \cite[\S6.3.2]{bachmann-eta}.
\begin{proof}
We first treat $\ko$.
Since very effective covers are stable under pro-smooth base change \cite[Lemma B.1]{bachmann-norms}, and $\ko^\fr$ is stable under arbitrary base change (see the proof of Theorem \ref{thm:main-1} below), we may assume that the base field is perfect.
In this case $\Spc^\fr(k)^\gp \wequi \SH(k)^\veff$ \cite[Theorem 3.5.14(i)]{EHKSY} and hence $\Sigma^\infty_\fr\Omega^\infty_\fr$ coincides with the very effective cover functor.
Thus $\ko \wequi \ko^\fr$ as needed.

The claim for $\kw$ follows via \cite[Lemma 6.9]{bachmann-eta} and Lemma \ref{lemm:identify-kw}(4).

For $\ul{K}^W, \ul{k}^M$, the claim is true essentially by construction.
This implies the claim for $\HW$ via Lemma \ref{lemm:identify-HW}.

The claim for $\HZW$ now follows from \cite[Theorem 17]{bachmann-very-effective} (see also \cite[(6.5)]{bachmann-eta}).
For $\HtildeZ$ consider the commutative diagram
\begin{equation*}
\begin{CD}
E @>>> \ul{K}^{MW} @>>> \ul{K}^W \\
@VVV     @VVV           @VVV     \\
\HZ @>>> \ul{K}^M @>>> \ul{k}^M.
\end{CD}
\end{equation*}
The left hand square is defined to be cartesian, so that by \cite[Theorem 17]{bachmann-very-effective}, $E$ coincides with the usual definition of $\HtildeZ$.
The right hand square is cartesian by \cite[Theorem 5.3]{morel2004puissances}.
Hence the outer rectangle is cartesian and $E \wequi \HtildeZ$ as defined above.
This concludes the proof.
\end{proof}

\begin{lemma} \label{lemm:kw-HW}
Over a Dedekind scheme containing $1/2$ we have $\ul{\pi}_*(\kw) \wequi \ul{W}[\beta]$ and $\kw/\beta \wequi \HW$.
\end{lemma}
\begin{proof}
Immediate from Remarks \ref{rmk:KW-htpy} and \ref{rmk:htpy-sheaves-conservative}.
\end{proof}

The following is one of our main results.
\begin{theorem} \label{thm:main-1}
The spectra $\ul{K}^W, \ul{k}^M, \HW, \HZW, \HtildeZ, \ko^\fr$ and $\kw$ are stable under base change among Dedekind schemes containing $1/2$.
\end{theorem}
\begin{proof}
For $\ko^\fr$ this follows from the facts that (1) $\Omega^\infty \KO \in \Spc(D)$ is motivically equivalent to the orthogonal Grassmannian \cite[Theorem 1.1]{schlichting2015geometric}, which is stable under under base change, and that (2) the forgetful functor $\Spc^\fr(D) \to \Spc(D)$ commutes with base change \cite[Lemma 16]{hoyois2018localization}.
The case of $\kw$ follows from this and Lemma \ref{lemm:identify-kw}, which shows that $\kw \wequi \ko^\fr[\eta^{-1}]$.
The case of $\HW$ now follows from Lemma \ref{lemm:kw-HW}.

The spectra $\HZ$ and $\HZ/2$ are stable under base change \cite{spitzweck2012commutative}, and hence so is $\ul{k}^M \wequi (\HZ/2)/\tau$ (see Lemma \ref{lemm:kM-mot-local}).
To show that $\ul{K}^W$ is stable under base change it suffices to show the same about $\ul{K}^W[\eta^{-1}]$ and $\ul{K}^W/\eta$ (see e.g. \cite[Lemma 2.16]{bachmann-eta}).
In light of Lemma \ref{lemm:identify-HW}, this thus follows from the cases of $\HW$ and $\ul{k}^M$, which we have already established.

Finally the stability under base change of $\HZW$ and $\HtildeZ$ reduces by definition to the same claim about $\HZ, \HZ/2, \ul{k}^M$ and $\ul{K}^W$, which we have aready dealt with.

This concludes the proof.
\end{proof}

\begin{definition}
Let $S$ be a scheme with $1/2 \in S$.
Define spectra $\ul{K}^W, \ul{k}^M, \HW, \HZW, \HtildeZ, \ko^\fr, \kw \in \SH(S)$ by pullback along the unique map $S \to \Spec(\Z[1/2])$.
\end{definition}
\begin{remark} \label{rmk:spectra-general}
\begin{enumerate}
\item By Theorem \ref{thm:main-1}, if $S$ is a Dedekind scheme, the new and old definitions agree.
\item The proof of Theorem \ref{thm:main-1} shows that for $S$ regular we have $\ko^\fr_S \wequi \Sigma^\infty_\fr \Omega^\infty_\fr \KO_S$ and hence $\kw_S \wequi (\Sigma^\infty_\fr \Omega^\infty_\fr \KO_S)[\eta^{-1}]$.
\item Since formation of homotopy sheaves is compatible with essentially smooth base change, and so is $\ul{W}$ (see e.g. \cite[Corollary 51]{bachmann-gwtimes}), we find that if $S \to D$ is essentially smooth then $\ul{\pi}_*(\kw_S) \wequi \ul{W}[\beta]$, and similarly for the homotopy sheaves of the other spectra.
\end{enumerate}
\end{remark}

\begin{lemma} \label{lemm:neg-crit}
Suppose that $S$ has finite dimension.
Let $E \in \SH(S)$ (respectively $E \in \SH(S)^\eff$).
Then $E \in \SH(S)_{\le 0}$ (respectively $E \in \SH(S)^\eff_{\le 0}$) if and only if $\ul{\pi}_i(E)_* = 0$ for $i > 0$ and $* \in \Z$ (respectively $i > 0$ and $*=0$).
\end{lemma}
\begin{proof}
We give the proof for $\SH(S)$, the one for $\SH(S)^\eff$ is analogous.
Since by definition $\Sigma^i \Sigma^\infty_+ X \wedge \Gmp{j} \in \SH(S)_{>0}$ for $i>0, j \in \Z$, necessity is straightforward.
We now show sufficiency.
Let $\scr C \subset \SH(S)$ denote the subcategory on those spectra $F$ with $\Map(\Sigma F, E) = *$.
We need to show that $\SH(S)_{\ge 0} \subset \scr C$.
By definition, for this it is enough to show that (a) $\scr C$ is closed under colimits, (b) $\scr C$ is closed under extensions, and (c) $\Sigma^\infty_+ X \wedge \Gmp{i} \in \scr C$ for every $X \in \Sm_S$ and $i \in \Z$.
(a) is clear, and (b) follows from the fact that given any fiber sequence of spaces $* \to S \to *$ we must have $S \wequi *$.
To prove (c), it is enough to show that $\Omega^\infty(\Gmp{i} \wedge E) \wequi * \in \Shv_{\Spc}^\Nis(\Sm_S)$.
By \cite[Proposition A.3]{bachmann-norms}, this can be checked on homotopy sheaves.
\end{proof}
\begin{remark}
In contrast with the case of fields, if $S$ has positive dimension, then homotopy sheaves do not characterize $\SH(S)_{\ge 0}$.
That is, given $E \in \SH(S)_{\ge 0}$, it need not be the case that $\ul\pi_i(E)_* = 0$ for $i < 0$.\footnote{This was previously known as \emph{Morel's stable connectivity conjecture}.}
\end{remark}

We denote by $\tau_{\ge 0}, \tau_{\le 0}$ (respectively $\tau_{\ge 0}^\eff, \tau_{\le 0}^\eff$) the truncation functors for the homotopy $t$-structure on $\SH(S)$ (respectively $\SH(S)^\eff$) and $f_0$ for the effective cover functor; see e.g. \cite[\S B]{bachmann-norms} for a uniform treatment.
\begin{remark} \label{rmk:check-over-fields}
We will repeatedly use \cite[Proposition B.3]{bachmann-norms}, which states the following: if $E \in \SH(S)$, where $S$ is finite dimensional, then $E \in \SH(S)_{\ge 0}$ (respectively $\SH(S)^\eff, \SH(S)^\eff_{\ge 0}$, respectively $\{0\}$) if and only if for every point $s \in S$ with inclusion $i_s: s \hookrightarrow S$ we have $i_s^*(E) \in \SH(s)_{\ge 0}$ (respectively $i_s^*(E) \in \SH(s)^\eff, i_s^*(E) \in \SH(s)^\eff_{\ge 0}$, $i_s^*(E) \wequi 0$).
\end{remark}
\begin{corollary} \label{cor:main1}
Let $D$ be a Dedekind scheme containing $1/2$, and $S \to D$ pro-smooth (e.g. essentially smooth).
The canonical maps exhibit equivalences in $\SH(S)$
\begin{gather*}
\kw \wequi \tau_{\ge 0} \KW \\
\HW \wequi \tau_{\le 0} \1[\eta^{-1}] \wequi \tau_{\le 0} \kw \\
\HZW \wequi f_0 \HW \wequi f_0 \ul{K}^W \\
\HtildeZ \wequi \tau_{\le 0}^\eff \1 \wequi \tau_{\le 0}^\eff \ko^\fr \\
\ul{K}^W \wequi \tau_{\le 0} \HZW.
\end{gather*}
\end{corollary}
\begin{proof}
By \cite[Lemma B.1]{bachmann-norms}, pro-smooth base change commutes with truncation in the homotopy $t$-structure, in the effective homotopy $t$-structure, and also with effective covers.
We may thus assume that $S=D$.

Consider the cofiber sequence $\kw \to \KW \to E$.
To prove that $\kw \wequi \tau_{\ge 0}\KW$, it is enough to show that $\kw \in \SH(D)_{\ge 0}$ and $E \in \SH(D)_{<0}$.
By Remark \ref{rmk:check-over-fields} we may check the first claim after base change to fields, and hence by Theorem \ref{thm:main-1} for this part we may assume that $D$ is the spectrum of a field, where the claim holds by definition.
The second claim follows via Lemma \ref{lemm:neg-crit} from our definition of $\kw$ as a connective cover in the Nisnevich topology (see Definition \ref{def:kw}).

Consider the fiber sequence $F \to \1[\eta^{-1}] \to \HW$.
To prove that $\tau_{\le 0} \1[\eta^{-1}] \wequi \HW$ it suffices to show that $F \in \SH(D)_{>0}$ and $\HW \in \SH(D)_{\le 0}$.
As before the first claim can be checked over fields where it holds by definition, and the second one follows from Lemma \ref{lemm:neg-crit} and the definition of $\HW$ (see Definition \ref{def:HW}).
The argument for $\tau_{\le 0} \kw \wequi \HW$ is similar.

Since the map $\ul{K}^W \to \HW$ of \S\ref{subsec:KW-HW} induces an isomorphism on $\ul{\pi}_*(\ph)_0$ (by construction), we have $f_0 \HW \wequi f_0 \ul{K}^W$.
We have $\HZW \in \SH(D)^\eff$ by checking over fields; it thus remains to show that $f_0 \HZW \wequi f_0 \ul{K}^W$.
By the defining fiber square, for this it is enough to show that $f_0 \HZ/2 \wequi f_0 \ul{k}^M$.
This follows from the cofibration sequence \[ \Sigma^{0,-1} \HZ/2 \xrightarrow{\tau} \HZ/2 \to \ul{k}^M \] together with the fact that $f_1 \HZ/2 \wequi 0$ (see e.g. \cite[Theorem B.4]{bachmann-norms}).

We have $\HtildeZ \in \SH(D)^\veff$ by checking over fields.
Knowledge of $\ul{\pi}_*(\HZ)_0 \wequi \Z$, $\ul{\pi}_*(\ul{k}^M)_0 = \Z/2$ and $\ul{\pi}_*(\ul{K}^W)_0 \wequi \ul{W}$ implies via Lemma \ref{lemm:neg-crit} that $\HtildeZ \in \SH(D)^\eff_{\le 0}$.
It hence remains to show that the fibers of $\1 \to \HtildeZ$ and $\ko^\fr \to \HtildeZ$ are in $\SH(D)^\eff_{>0}$, which may again be checked over fields.

Consider the fiber sequence $F \to \HZW \to \ul{K}^W$.
We need to show that $F \in \SH(D)_{\ge 1}$ and $\ul{K}^W \in \SH(D)_{\le 0}$.
As before the first claim can be checked over fields where it holds by \cite[Lemma 18]{bachmann-very-effective}, and the second claim follows from Lemma \ref{lemm:neg-crit} and the knowledge of the homotopy sheaves of $\ul{K}^W$, i.e. Remark \ref{rmk:W-contraction}.

This concludes the proof.
\end{proof}

\subsection{}
Recall from \cite[\S3]{bachmann-eta} the stable Adams operation $\adamspsi^3: \KO[1/3] \to \KO[1/3] \in \SH(D)$.
Via Corollary \ref{cor:main1} this induces $\adamspsi^3: \kw_{(2)} \to \kw_{(2)} \in \SH(D)$.

\begin{theorem} \label{thm:main2}
Let $D$ be a Dedekind scheme with $1/2 \in D$.
\begin{enumerate}
\item The map $\adamspsi^3 - \id: \kw_{(2)} \to \kw_{(2)}$ factors uniquely (up to homotopy) through $\beta: \Sigma^4 \kw_{(2)} \to \kw_{(2)}$, yielding \[ \adamsphi: \kw_{(2)} \to \Sigma^4 \kw_{(2)} \in \SH(D). \]
\item The unit map $\1 \to \kw_{(2)}$ factors uniquely (up to homotopy) through the fiber of $\adamsphi$.
\item The resulting sequence \[ \1[\eta^{-1}]_{(2)} \to \kw \xrightarrow{\adamsphi} \Sigma^4 \kw_{(2)} \in \SH(D) \] is a fiber sequence.
\end{enumerate}
In particular for any scheme $S$ with $1/2 \in S$ there is a canonical fiber sequence \[ 1[\eta^{-1}]_{(2)} \to \kw_{(2)} \xrightarrow{\adamsphi} \Sigma^4 \kw_{(2)} \in \SH(S). \]
\end{theorem}
\begin{proof}
(1,2) We can repeat the arguments from \cite[Corollary 7.2]{bachmann-eta}.
It suffices to show that (a) $\1[\eta^{-1}] \to \kw$ is $1$-connected, (b) $\kw/\beta \in \SH(D)_{\le 0}$, and (c) $[\1, \Sigma^n \kw]_{\SH(D)} = 0$ for $n \in \{3,4\}$.
(a) can be checked over fields, hence holds by \cite[Lemma 7.1]{bachmann-eta}.
(b) follows from Lemma \ref{lemm:kw-HW} (showing that $\kw/\beta \wequi \HW$) and Corollary \ref{cor:main1} (showing that $\HW \in \SH(D)_{\le 0}$).
(c) follows from knowledge of the homotopy sheaves of $\kw$ together with the descent spectral sequence, using that $\dim D \le 1$.

(3) We need to show that the (unique) map $\1[\eta^{-1}] \to \fib(\adamsphi)$ is an equivalence.
This can be checked over fields, where it is \cite[Theorem 7.8]{bachmann-eta}.

The last claim follows by pullback along $S \to \Spec(\Z[1/2])$.
\end{proof}

\section{Applications} \label{sec:app}
Throughout we assume that $2$ is invertible on all schemes.
We shall employ the special linear and symplectic cobordism spectra $\MSL$ and $\MSp$; see e.g. \cite[Example 16.22]{bachmann-norms} for a definition.

\subsection{}
Recall that $\kw_* \MSL \wequi \kw_*[e_2, e_4, \dots]$ with $|e_{2i}| = 4i$ \cite[Theorem 4.1]{bachmann-eta}.
The canonical orientation $\MSL \to \kw$ induces $\MSL \to \kw \to \tau_{\le 0} \kw \wequi \HW$ and hence \[ \kw_* \MSL \to \kw_* \kw \to \kw_* \HW. \]
\begin{proposition}
The images of the $e_{2i}$ induce equivalences of right modules \[ \kw \wedge \kw_{(2)} \wequi \bigvee_{i \ge 0} \kw_{(2)}\{e_{2i}\} \quad\text{and}\quad \kw \wedge \HW_{(2)} \wequi \bigvee_{i \ge 0} \HW_{(2)}\{e_{2i}\}. \]
\end{proposition}
\begin{proof}
We may assume that $S = \Spec(\Z[1/2])$ and thus we may check that the induced map is an equivalence after base change to fields \cite[Proposition B.3]{bachmann-norms}.
Hence this follows from \cite[Proposition 7.7]{bachmann-eta}.
\end{proof}

\subsection{}
\begin{proposition} \label{prop:sphere-htpy-sheaves}
Let $S$ be essentially smooth over a Dedekind scheme.
We have \[ \ul{\pi}_*(\1_S[\eta^{-1}]) \wequi \begin{cases} \ul{W} & *=0 \\ \ul{W}[1/2] \otimes \pi_*^s \oplus \mathrm{coker}(8n: \ul{W}_{(2)} \to \ul{W}_{(2)}) & *=4n-1>0 \\ \ul{W}[1/2] \otimes \pi_*^s \oplus \ker(8n: \ul{W}_{(2)} \to \ul{W}_{(2)}) & *=4n>0 \\ \ul{W}[1/2] \otimes \pi_*^s & \text{else} \end{cases}. \]
Here $\pi_*^s$ denotes the classical stable stems.
\end{proposition}
\begin{proof}
We first show that $\ul{\pi}_0(\1[\eta^{-1}]) \to \ul{\pi}_0(\HW) \wequi \ul{W}$ is an isomorphism.
We may do so after $\otimes \Z_{(2)}$ and $\otimes \Z[1/2]$; the former case is immediate from the fundamental fiber sequence (using Remark \ref{rmk:spectra-general}(3)) and the latter case follows from real realization (since $\ul{W}[1/2] \wequi a_\ret \Z[1/2]$ \cite[Corollary 7.1]{jacobson-fundamental-ideal}).
In particular it follows that $\adamsphi$ is $\ul{W}$-linear (see \cite[Example 3.7]{bachmann-eta}).
The proof of \cite[Theorem 8.1]{bachmann-eta} now goes through unchanged.
\end{proof}

\subsection{}
\begin{lemma} \label{lemm:W-coh-vanishing}
Suppose that $D$ is a localization of $\Z$.
Then $H^*(D, \ul{W}) = 0$ for $*>0$.
\end{lemma}
\begin{proof}
\cite[Corollary IV.3.3]{milnor1973symmetric} shows that for any Dedekind scheme $D$ there is a natural exact sequence of abelian groups $(*)$ \[ 0 \to W(D) \to \bigoplus_{x \in D^{(0)}} W(x) \to \bigoplus_{x \in D^{(1)}} W(x, \omega_{x/D}). \]
The last map is surjective if $D$ has only one point, and hence these sequences constitute a resolution of the presheaf $W$ on $D_\Nis$.
The terms of this resolution are acyclic \cite[Lemma 5.42]{A1-alg-top}\footnote{The proof of this result does not use the stated assumption that $X$ is smooth over a field.}\NB{really?}, and hence this resolution can be used to compute cohomology.
In order to show that $H^*(D, \ul{W}) = 0$ for $*>0$ it suffices to show that the right most map of $(*)$ is surjective.
Clearly if this is true for $D$ then it also holds for any localization of $D$.
It hence suffices to prove surjectivity for $\Z$; this is \cite[Theorem IV.2.1]{milnor1973symmetric}.
\end{proof}

\begin{example}
Lemma \ref{lemm:W-coh-vanishing} shows that the descent spectral sequence for $\pi_*\kw$ (or $\pi_* \KW$) collapses over $\Z[1/2]$ (and localizations of this base).
Thus \[ \pi_* \kw_{\Z[1/2]} \wequi \W(\Z[1/2])[\beta]. \]
From this we can read off as in the proof of Proposition \ref{prop:sphere-htpy-sheaves} that (using that $\W(\Z[1/2])[1/2] \wequi \Z[1/2]$) \[ \pi_*(\1_{\Z[1/2]}[\eta^{-1}]) \wequi \begin{cases} \W(\Z[1/2]) & *=0 \\ \pi_*^s[1/2] \oplus \mathrm{coker}(8n: \W(\Z[1/2])_{(2)} \to \W(\Z[1/2])_{(2)}) & *=4n-1>0 \\ \pi_*^s[1/2] \oplus \ker(8n: \W(\Z[1/2])_{(2)} \to \W(\Z[1/2])_{(2)}) & *=4n>0 \\ \pi_*^s[1/2] & \text{else} \end{cases}. \]
\end{example}

\begin{remark} \label{rmk:W-Z-12}
We often use the above result in conjunction with the isomorphism (see e.g. \cite[proof of Theorem 5.11]{bachmann-euler}) \[ \W(\Z[1/2]) \wequi \Z[g]/(g^2, 2g). \]
(Here $g=\lra{2}-1$.)
Note in particular that $\W(\Z[1/2])_\red \wequi \Z$, $\W(\Z[1/2])_{(2)}$ is a local ring, and $\W(\Z[1/2]) \hookrightarrow \W(\Q)$.
\end{remark}

\subsection{}
\begin{proposition} \label{prop:MSL-MSp-kw}
\begin{enumerate}
\item We have \[ \pi_* \MSp_{\Z[1/2]}[\eta^{-1}] \wequi \W(\Z[1/2])[y_1, y_2, \dots]. \]
\item The generators from (1) induce for a scheme $S$ essentially smooth over a Dedekind scheme $D$ with $1/2 \in D$ \[ \ul{\pi}_* \MSp_S[\eta^{-1}] \wequi \ul{W}[y_1, y_2, \dots] \] and \[ \ul{\pi}_* \MSL_S[\eta^{-1}] \wequi \ul{W}[y_2, y_4, \dots]. \]
\item Over any $S$ we have $\MSp/(y_1, y_3, \dots) \wequi \MSL$.
\item There exist generators $y_2, y_4, \dots$ such that $\MSL/(y_4, y_6, \dots) \wequi \kw.$
\end{enumerate}
\end{proposition}
\begin{proof}
We implicitly invert $\eta$ throughout.

We first prove (1) and the part of (2) about $\MSp$, localized at $2$.
Note that $\kw_* \MSp_{(2)} \wequi \W(\Z[1/2])_{(2)}[e_1, e_2, \dots]$ is degreewise finitely generated over the local ring $\W(\Z[1/2])_{(2)}$ (see Remark \ref{rmk:W-Z-12}).
Moreover base change along $\Spec(\C) \to \Spec(\Z[1/2])$ implements the map $\W(\Z[1/2])_{(2)} \to \W(\C)$ of passing to the residue field.
Consider the morphism \[ \adamsphi: \kw_* \MSp_{(2)} \to \kw_{*-4} \MSp_{(2)}. \]
We deduce from the above discussion and \cite[Lemma 8.4]{bachmann-eta} that $\adamsphi \otimes_{\W(\Z[1/2])_{(2)}} \F_2$ is surjective, and hence $\adamsphi$ is split surjective.
Then $\ker(\adamsphi) \otimes_{\W(\Z[1/2])_{(2)}} \F_2 \wequi \ker(\adamsphi \otimes_{\W(\Z[1/2])_{(2)}} \F_2)$ is a polynomial ring on generators $\bar y_i$.
Write $\tilde y_i \in \pi_* \MSp_{(2)}$ for arbitrary lifts of these generators.
The proof of \cite[Corollary 8.6]{bachmann-eta} shows that if $\Spec(k) \to \Spec(\Z[1/2])$ is an arbitrary field, then $\pi_*(\MSp_{(2)})(k) \wequi \W(k)_{(2)}[\tilde y_1, \tilde y_2, \dots]$.
We shall show that if $S$ is henselian local and essentially smooth over $D$ then $\pi_*(\MSp_{(2)})(S) \wequi \pi_*(\MSp_{(2)})(s)$, where $s$ is the closed point; this will imply our claims.
Using the fundamental fiber sequence, for this it is enough to show that $\W(S) \wequi \W(s)$, which holds by \cite[Lemma 4.1]{jacobson2018cohomological}.

Note that by real realization, $\pi_* (\MSp[1/2])(\Z[1/2]) \wequi \pi_* (\MSp[1/2])(\R) \wequi \Z[y_1', y_2', \dots]$.
Let $J = \ker(\pi_*(\MSp)(\Z[1/2]) \to \pi_*(\HW)(\Z[1/2]))$ and $M=(J/J^2)_{2n}$.
We shall prove that $M \wequi \W(\Z[1/2])$.
The argument is the same as in \cite[Theorem 8.7]{bachmann-eta}: such isomorphisms exist for $M_{(2)}$ and $M[1/2]$ by what we have already done, which implies that $M$ is an invertible $\W(\Z[1/2])$-module, but $\W(\Z[1/2])_\red \wequi \Z$ (see Remark \ref{rmk:W-Z-12}) and hence all such invertible modules are trivial.
Write $y_n \in \pi_{2n}(\MSp)(\Z[1/2])$ for any lift of a generator of $M$.
We shall prove that the proposition holds with these choices of $y_i$.

Let $S$ be as in (2).
To show that $\ul{\pi}_*(\MSp_S) \wequi \ul{W}[y_1, y_2, \dots]$, it is enough that the map is an isomorphism after $\otimes \Z_{(2)}$ and after $\otimes \Z[1/2]$.
The former case was already dealt with.
For the latter, we use the equivalence $\SH(S)[1/2, 1/\eta] \wequi \SH(S_\ret)[1/2]$.
We are this way reduced to proving that if $E \in \SH(\Z[1/2]_\ret)[1/2]$ with $\pi_*E \wequi \Z[1/2, y_1, y_2, \dots]$, then for any morphism $f: S \to \Z[1/2]$ we have $\ul{\pi}_*(f^*E) \wequi \ul{W}[1/2, y_1, y_2, \dots]$.
This follows from the facts that (a) $\Sper(\Z[1/2]) \wequi *$, so our assumption implies that $\ul{\pi}_*E \wequi a_\ret \Z[1/2, y_1, y_2, \dots]$, (b) $f^*$ is (in this setting) $t$-exact, and (c) $\ul{W}[1/2] \wequi a_\ret \Z[1/2]$.

We have now proved (2) for $\MSp$.
Claim (1) follows from the descent spectral sequence and Lemma \ref{lemm:W-coh-vanishing}.
Next we claim that over $\Z[1/2]$ (and hence in general), $\MSp \to \MSL$ annihilates $y_i$ for $i$ odd.
Indeed we can check this after $\otimes \Z_{(2)}$ and $\otimes \Z[1/2]$; the latter cases reduces via real realization to the base field $\R$ where we already know it.
For the former case, we first verify as above that $\adamsphi: \kw_* \MSL_{(2)} \to \kw_{*-4} \MSL_{(2)}$ is surjective, and hence $\pi_* \MSL_{(2)} \to \kw_* \MSL_{(2)}$ is injective; the claim follows easily from this.
It follows that we may form $\MSp/(y_1, y_3, \dots) \to \MSL$.
To check that this is an equivalence we may pull back to fields, and so we are reduced to \cite[Corollary 8.9]{bachmann-eta}.
We have now proved (3), which implies the missing half of (2).

It remains to establish (4).
We claim that $\pi_4 \MSL_{\Z[1/2]} \to \pi_4 \kw_{\Z[1/2]}$ is surjective, and hence an isomorphism since both groups are free $\W(\Z[1/2])$-modules of rank $1$.
This we may check after $\otimes \Z_{(2)}$ and $\otimes \Z[1/2]$.
In the former case, we have a morphism of finitely generated modules over a local ring, so may check $\otimes_{\W(\Z[1/2])_{(2)}} \F_2$, i.e. over $\C$, in which case the claim holds by \cite[Lemma 8.10]{bachmann-eta}.
In the latter case, via real realization we reduce to $\R$, and so again the claim holds by \cite[Lemma 8.10]{bachmann-eta}.
The upshot is that we may choose $y_2$ in such a way that its image in $\pi_4 \kw$ is the generator $\beta$.
Then as in the proof of \cite[Corollary 8.11]{bachmann-eta} we modify the other generators to be annihilated in $\kw$ to obtain a map $\MSL/(y_4, y_6, \dots) \to \kw$.
This map is an equivalence since it is so after pullback to fields, by \cite[Corollary 8.11]{bachmann-eta}.
\end{proof}

\subsection{}
\begin{proposition}
Let $\scr S$ be the set of primes not invertible on $S$.
The spectra $\kw, \HW, \HtildeZ[\scr S^{-1}], \HZW, \ul{K}^W \in \SH(S)$ are cellular.
\end{proposition}
\begin{proof}
We can argue as in \cite[Proposition 8.12]{bachmann-eta}:

For $\kw$ this follows from Proposition \ref{prop:MSL-MSp-kw}(3, 4) and cellularity of $\MSp$.
Hence $\HW \wequi \kw/\beta$ is cellular.
$\HZ[\scr S^{-1}]$ is cellular by \cite[Corollary 11.4]{spitzweck2012commutative}; in particular $\HZ/2$ and $\ul{k}^M \wequi (\HZ/2)/\tau$ are cellular.
Thus $\ul{K}^W/\eta \wequi \ul{k}^M$ and $\ul{K}^W[\eta^{-1}] \wequi \HW$ are cellular and thus so is $\ul{K}^W$.
Cellularity of $\HtildeZ[\scr S^{-1}]$ and $\HZW$ now follows from the defining fiber squares (in which all the other objects are cellular by what we have already proved).
\end{proof}

\subsection{}
\begin{proposition} \label{prop:kw-op}
For arbitrary $S$ we have \[ \kw^*_{(2)}\kw \wequi \kw^*_{(2)}\fpsr{'\adamsphi}, \] in the sense that $\adamsphi$ need not be central and so the multiplicative structure is more complicated than a power series ring.
We have $\adamsphi\beta = 9\beta \adamsphi + 8$.
\end{proposition}
\begin{proof}
Stability under base change implies that \cite[Lemma 8.14]{bachmann-eta} holds over any base; hence the additive structure of $\kw^*_{(2)}\kw$ can be determined as in \cite[proof of Corollary 8.15]{bachmann-eta}.
The interaction of $\beta$ and $\adamsphi$ may be determined over $\Z[1/2]$.
Since $\W(\Z[1/2]) \hookrightarrow \W(\Q)$ we are reduced to $S=\Spec(\Q)$, which was already dealt with in \cite[Corollary 8.19]{bachmann-eta}.
\end{proof}
\begin{example}
Suppose that $S$ is essentially smooth over a Dedekind scheme, $H^*(S, \ul{W}) = 0$ for $*>0$, and $\W(S)$ is generated by $1$-dimensional forms (e.g. $S$ local or a localization of $\Z$).
Then $\kw^* \wequi \W(S)[\beta]$ and $\pi_0(\1) \to \pi_0(\kw)$ is surjective, whence $\adamsphi$ commutes with $\W(S)$.
It follows that Proposition \ref{prop:kw-op} yields a complete description of $\kw^*_{(2)}\kw$.
\end{example}

\subsection{}
\begin{proposition}
\begin{enumerate}
\item There exist generators $x_i \in \pi_{4i}(\kw \wedge \HW_{(2)})$ such that \[ x_mx_n = {m+n \choose n}x_{m+n}. \]
\item We have \[ \HW \wedge \HW_{(2)} \wequi \bigvee_{n \ge 0} \Sigma^{4n} \HW/8n. \]
\end{enumerate}
\end{proposition}
\begin{proof}
(1) The proofs of \cite[Propositions 8.16 and 8.18]{bachmann-eta} can be repeated unchanged.

(2) The image of $\beta$ in $\pi_4(\kw \wedge \HW_{(2)})$ is given by $ax_1$ for some $a \in \W(\Z[1/2])$.
The argument of \cite[Lemma 8.14]{bachmann-eta} shows that $a=8u$ for some unit $u$.
The result now follows from (1), as in the proof of \cite[Corollary 8.16]{bachmann-eta}.
\end{proof}

\bibliographystyle{amsalphac}
\bibliography{bibliography}

\newcommand{\etalchar}[1]{$^{#1}$}
\providecommand{\bysame}{\leavevmode\hbox to3em{\hrulefill}\thinspace}
\providecommand{\MR}{\relax\ifhmode\unskip\space\fi MR }
\providecommand{\MRhref}[2]{%
  \href{http://www.ams.org/mathscinet-getitem?mr=#1}{#2}
}
\providecommand{\href}[2]{#2}
\begin{thebibliography}{EHK{\etalchar{+}}21}

\bibitem[AR{\O}20]{ananyevskiy2017very}
Alexey Ananyevskiy, Oliver R{\"o}ndigs, and Paul~Arne {\O}stv{\ae}r, On very
  effective hermitian K-theory, Mathematische Zeitschrift \textbf{294} (2020),
  no.~3, 1021--1034.

\bibitem[Bac17]{bachmann-very-effective}
Tom Bachmann, The generalized slices of Hermitian K-theory, Journal of Topology
  \textbf{10} (2017), no.~4, 1124--1144,
  \href{https://arxiv.org/abs/1610.01346}{arXiv:1610.01346}.

\bibitem[Bac18a]{bachmann-real-etale}
\bysame, Motivic and real étale stable homotopy theory, Compositio Mathematica
  \textbf{154} (2018), no.~5, 883–917,
  \href{https://arxiv.org/abs/1608.08855}{arXiv:1608.08855}.

\bibitem[Bac18b]{bachmann-gwtimes}
Tom Bachmann, Some remarks on units in Grothendieck–Witt rings, Journal of
  Algebra \textbf{499} (2018), 229 -- 271,
  \href{https://arxiv.org/abs/1612.04728}{arXiv:1707.08087}.

\bibitem[Bal05]{balmer2005witt}
P~Balmer, Witt groups. Handbook of K-theory, 539-576, 2005.

\bibitem[BF17]{bachmann-criterion}
Tom Bachmann and Jean Fasel, On the effectivity of spectra representing motivic
  cohomology theories,
  \href{https://arxiv.org/abs/1710.00594}{arXiv:1710.00594}.

\bibitem[BH20]{bachmann-eta}
Tom Bachmann and Michael~J. Hopkins, $\eta$-periodic motivic stable homotopy
  theory over fields,
  \href{https://arxiv.org/abs/2005.06778}{arXiv:2005.06778}, 2020.

\bibitem[BH21]{bachmann-norms}
Tom Bachmann and Marc Hoyois, Norms in Motivic Homotopy Theory, Ast{\'e}risque
  \textbf{425} (2021),
  \href{https://arxiv.org/abs/1711.03061}{arXiv:1711.03061}.

\bibitem[BW21]{bachmann-euler}
Tom Bachmann and Kirsten Wickelgren, $\mathbb{A}^1$-Euler classes: six functors
  formalisms, dualities, integrality and linear subspaces of complete
  intersections, Journal of the Institute of Mathematics of Jussieu (2021),
  \href{https://arxiv.org/abs/2002.01848}{arXiv:2002.01848}.

\bibitem[CD09]{cisinski2009local}
Denis-Charles Cisinski and Fr{\'e}d{\'e}ric D{\'e}glise, Local and stable
  homological algebra in Grothendieck abelian categories, Homology, Homotopy
  and Applications \textbf{11} (2009), no.~1, 219--260.

\bibitem[DF17]{Deglise17}
F.~D{\'e}glise and J.~Fasel, {MW}-motivic complexes, arxiv:1708.06095.

\bibitem[DKY22]{deshmukh2020nisnevich}
Neeraj Deshmukh, Girish Kulkarni, and Suraj Yadav, A Nisnevich local Bloch-Ogus
  theorem over a general base, Journal of Pure and Applied Algebra \textbf{226}
  (2022), no.~6, 106978.

\bibitem[EHK{\etalchar{+}}20]{EHKSY2}
Elden Elmanto, Marc Hoyois, Adeel~A Khan, Vladimir Sosnilo, and Maria Yakerson,
  Framed transfers and motivic fundamental classes, Journal of Topology
  \textbf{13} (2020), no.~2, 460--500.

\bibitem[EHK{\etalchar{+}}21]{EHKSY}
Elden Elmanto, Marc Hoyois, Adeel~A. Khan, Vladimir Sosnilo, and Maria
  Yakerson, Motivic infinite loop spaces, Cambridge Journal of Mathematics
  \textbf{9} (2021), no.~2, 431--549.

\bibitem[GAV72]{sga4}
Alexander Grothendieck, M~Artin, and JL~Verdier, Th{\'e}orie des topos et
  cohomologie {\'e}tale des sch{\'e}mas ({SGA} 4), Lecture Notes in Mathematics
  \textbf{269} (1972).

\bibitem[Gei04]{geisser2004motivic}
Thomas Geisser, Motivic cohomology over Dedekind rings, Mathematische
  Zeitschrift \textbf{248} (2004), no.~4, 773--794.

\bibitem[{Gro}67]{EGAIV}
A.~{Grothendieck}, {\'El\'ements de g\'eom\'etrie alg\'ebrique. IV: \'Etude
  locale des sch\'emas et des morphismes de sch\'emas. R\'edig\'e avec la
  colloboration de Jean Dieudonn\'e.}, {Publ. Math., Inst. Hautes \'Etud. Sci.}
  \textbf{32} (1967), 1--361 (French).

\bibitem[Hor05]{hornbostel2005a1}
Jens Hornbostel, $A^1$-representability of Hermitian K-theory and Witt groups,
  Topology \textbf{44} (2005), no.~3, 661--687.

\bibitem[Hoy15]{hoyois-algebraic-cobordism}
Marc Hoyois, From algebraic cobordism to motivic cohomology, Journal f{\"u}r
  die reine und angewandte Mathematik (Crelles Journal) \textbf{2015} (2015),
  no.~702, 173--226.

\bibitem[Hoy16]{hoyois2016equivariant}
\bysame, Cdh descent in equivariant homotopy $K$-theory, arXiv preprint
  arXiv:1604.06410 (2016).

\bibitem[Hoy21]{hoyois2018localization}
\bysame, The localization theorem for framed motivic spaces, Compositio
  Mathematica \textbf{157} (2021), no.~1, 1--11.

\bibitem[ILO14]{illusie2014travaux}
Luc Illusie, Yves Laszlo, and Fabrice Orgogozo, Travaux de Gabber sur
  l'uniformisation locale et la cohomologie {\'e}tale des sch{\'e}mas
  quasi-excellents. S{\'e}minaire {\`a} l'{\'E}cole polytechnique 2006--2008,
  2014.

\bibitem[Jac17]{jacobson-fundamental-ideal}
Jeremy Jacobson, Real cohomology and the powers of the fundamental ideal in the
  Witt ring, Annals of K-Theory \textbf{2} (2017), no.~3, 357--385.

\bibitem[Jac18]{jacobson2018cohomological}
Jeremy~Allen Jacobson, Cohomological invariants for quadratic forms over local
  rings, Mathematische Annalen \textbf{370} (2018), no.~1-2, 309--329.

\bibitem[Kne77]{knebusch-bilinear}
Manfred Knebusch, Symmetric bilinear forms over algebraic varieties, Conference
  on quadratic forms (G.~Orzech, ed.), Queen's papers in pure and applied
  mathematics, vol.~46, Queens University, Kingston, Ontario, 1977,
  pp.~103--283.

\bibitem[Lur16]{lurie-ha}
Jacob Lurie, Higher Algebra, May 2016.

\bibitem[Lur18]{lurie-sag}
\bysame, Spectral Algebraic Geometry, February 2018.

\bibitem[MH73]{milnor1973symmetric}
John~Willard Milnor and Dale Husemoller, Symmetric bilinear forms, vol.~60,
  Springer, 1973.

\bibitem[Mor04a]{morel2004motivic-pi0}
Fabien Morel, On the motivic $\pi_0$ of the sphere spectrum, Axiomatic,
  enriched and motivic homotopy theory, Springer, 2004, pp.~219--260.

\bibitem[Mor04b]{morel2004puissances}
\bysame, Sur les puissances de l’id{\'e}al fondamental de l’anneau de Witt,
  Commentarii Mathematici Helvetici \textbf{79} (2004), no.~4, 689--703.

\bibitem[Mor12]{A1-alg-top}
\bysame, $\mathbb{A}^1$-Algebraic Topology over a Field, Lecture Notes in
  Mathematics, Springer Berlin Heidelberg, 2012.

\bibitem[PW10]{panin2010motivic}
Ivan Panin and Charles Walter, On the motivic commutative ring spectrum BO,
  arXiv preprint arXiv:1011.0650 (2010).

\bibitem[Sch17]{schlichting2016hermitian}
Marco Schlichting, Hermitian K-theory, derived equivalences and Karoubi's
  fundamental theorem, Journal of Pure and Applied Algebra \textbf{221} (2017),
  no.~7, 1729--1844.

\bibitem[Seg74]{segal1974categories}
Graeme Segal, Categories and cohomology theories, Topology \textbf{13} (1974),
  no.~3, 293--312.

\bibitem[Ser13]{serre2013galois}
Jean-Pierre Serre, Galois cohomology, Springer Science \& Business Media, 2013.

\bibitem[Spi18]{spitzweck2012commutative}
Markus Spitzweck, A commutative $\mathbb{P}^1$-spectrum representing motivic
  cohomology over {D}edekind domains, M{\'e}m. Soc. Math. Fr. \textbf{157}
  (2018).

\bibitem[ST15]{schlichting2015geometric}
Marco Schlichting and Girja~S Tripathi, Geometric models for higher
  Grothendieck--Witt groups in $\mathbb{A}^1$-homotopy theory, Mathematische
  Annalen \textbf{362} (2015), no.~3-4, 1143--1167.

\bibitem[{Sta}18]{stacks-project}
The {Stacks Project Authors}, {\itshape Stacks Project},
  \url{http://stacks.math.columbia.edu}, 2018.

\end{thebibliography}

\end{document}